\documentclass[a4paper, 12pt]{amsart}
\usepackage{amsmath,amssymb,amsthm}
\usepackage{amscd}
\usepackage{array,enumerate}
\newtheorem{Thm}{Theorem}[section]

\newtheorem{Lem}[Thm]{Lemma}

\newtheorem{Thmint}{Theorem}[section]
\newtheorem{MThm}{Main Theorem}

\newtheorem{Corint}[Thmint]{Corollary}

\theoremstyle{definition}
\newtheorem{Rem}[Thm]{Remark}
\newtheorem{Defint}[Thmint]{Definition}
\newtheorem{Exmint}[Thmint]{Example}
\newtheorem{Def}[Thm]{Definition}

\newcommand{\Cs}{C$^\ast$}

\newcommand{\ip}[1]{\mathopen{\langle}#1\mathclose{\rangle}}
\newcommand{\id}{\mbox{\rm id}}

\newcommand{\rg}{\mathop{{\mathrm C}_{\mathrm r}^\ast}}

\newcommand{\rc}{\mathop{\rtimes _{\mathrm r}}}

\newcommand{\rca}[1]{\mathop{\rtimes _{{\mathrm r}, #1}}}

\newcommand{\Ped}[1]{\mathrm{Ped}({#1})}
\newcommand{\Her}[1]{\mathrm{Her}({#1})}
\newcommand{\cl}[1]{\mathrm{cl}({#1})}

\newcommand{\IB}{\mathbb B}

\newcommand{\IK}{\mathbb K}

\newcommand{\IN}{\mathbb N}

\newcommand{\cI}{\mathcal I}
\newcommand{\cJ}{\mathcal J}
\newcommand{\cU}{\mathcal U}
\newcommand{\fH}{\mathfrak{H}}
\DeclareMathOperator{\stlim}{strict-}

\newcommand{\cE}{\mathcal E}

\newcommand{\cM}{\mathcal M}
\DeclareMathOperator{\supp}{supp}

\DeclareMathOperator{\Cso}{{\rm C}^\ast}

\DeclareMathOperator{\bigfp}{\lower0.25ex\hbox{\LARGE $\ast$}}
\newcommand{\ad}{\mathop{\rm Ad}}
\title[Simplicity and tracial weights on non-unital crossed products]{Simplicity and tracial weights on non-unital reduced crossed products}
\author{Yuhei Suzuki}
\address{Department of Mathematics, Faculty of Science, Hokkaido University,
Kita 10, Nishi 8, Kita-Ku, Sapporo, Hokkaido, 060-0810, Japan}
\email{yuhei@math.sci.hokudai.ac.jp}
\begin{document}
\maketitle
\begin{abstract}
We extend theorems of Breuillard--Kalantar--Kennedy--Ozawa on unital reduced crossed products
to the non-unital case under mild assumptions.
As a result simplicity of \Cs-algebras is stable under taking reduced crossed products over discrete \Cs-simple groups, and a similar result for uniqueness of tracial weight.
Interestingly, our analysis on tracial weights involves von Neumann algebra theory.

Our generalizations have two applications.
The first is to locally compact groups.
We establish stability results of (non-discrete) C$^\ast$-simplicity and the unique trace property under discrete group extensions.
The second is to the twisted crossed product.
Thanks to the Packer--Raeburn theorem, our results lead to (generalizations of)
the results of Bryder--Kennedy by a different method.
\end{abstract}
\section{Introduction}
In this article, we aim to extend the remarkable theorems of Breuillard--Kalantar--Kennedy--Ozawa \cite{BKKO} on simplicity and classification of tracial states
of the \emph{unital} reduced crossed product \Cs-algebras
to the \emph{non-unital} case.
Besides technical intrinsic interests, our generalization
has two applications: Stability of \Cs-simplicity and the unique trace property
of \emph{non-discrete} locally compact groups by discrete group extensions,
and a new simpler approach (and generalizations)
of the results of Bryder--Kennedy \cite{BK} on unital reduced twisted crossed product \Cs-algebras. 

In the original proofs of \cite{BKKO} (see also \cite{Haa}, \cite{KK} for alternative approaches),
the unit of the underlying \Cs-algebra plays prominent roles at several key steps.
The main goal of this work is thus to find suitable replacements of the unit.

For the simplicity result, we employ the following elements as a replacement of the unit.
We say an element $a\in (A_+)_1$ is
\emph{completely full} if $f(a)$ is full for all functions $f\in C_0(0, 1]_+$ with $f(1)>0$.
(Here we recall that a positive element $a\in A_+$ is said to be
\emph{full} if the (closed) ideal generated by $a$ is equal to $A$.)

Observe that $A$ admits a completely full element if (and only if) it admits
a full element of the form
$f(a)$; $a \in (A_+)_1$, $f\in C_c(0, 1]_+$.
Indeed for such $a$ and $f$, choose a continuous function $g\in (C_0(0, 1]_+)_1$
satisfying $g\equiv 1$ on $\supp(f)$.
Then $g(a)$ is completely full.
Indeed for any $h\in C_0(0, 1]$ we have
$h(g(a))f(a)=h(1)f(a)$.
\begin{Exmint}
Here we give a list of basic examples of \Cs-algebras which admit a completely full element.
\begin{enumerate}[\upshape(1)]
\item For a simple \Cs-algebra $A$,
any nonzero positive element is full.
Thus any positive element of norm one is completely full.
\item For a unital \Cs-algebra $A$,
the stabilization $A \otimes \IK(\fH)$ admits a completely full element.
Indeed any nonzero positive element in $\IK(\fH)$ is full in $A \otimes \IK(\fH)$.
\item More generally, if both two \Cs-algebras $A, B$ admit a completely full element,
then so does the minimal tensor product $A \otimes B$.
Indeed take completely full elements $a\in (A_+)_1$, $b\in (B_+)_1$.
Choose functions $f, g \in (C_c(0, 1]_+)_1$ satisfying
\[f(1)=1,\quad \supp(f) \subset [1/2, 1],\quad g\equiv 1 {\rm~on~} [1/4, 1].\]
Then as $g(st) f(s)f(t)=f(s)f(t)$ for all $s, t\in [0, 1]$, it follows from the Gelfand--Naimark theorem
that
\[g(a\otimes b)\cdot (f(a)\otimes f(b)) =f(a) \otimes f(b).\]
Since $f(a)\otimes f(b)$ is full, so is $g(a\otimes b)$.
Hence $g(a\otimes b)$ is completely full in $A \otimes B$.
\end{enumerate}
\end{Exmint}
Our simplicity theorem is stated as follows.
\begin{MThm}\label{MThm:simple}
Let $\Gamma$ be a discrete \Cs-simple group.
Let $A$ be a \Cs-algebra with a completely full element.
Let $(\alpha, u) \colon \Gamma \curvearrowright A$ be a twisted action.
Then the reduced twisted crossed product \Cs-algebra $A \rca{\alpha, u} \Gamma$ is simple if and only if
$A$ has no proper $\Gamma$-invariant ideal.
\end{MThm}
Note that in the unital case (which is the case established by Bryder--Kennedy \cite{BK} in the twisted case),
one can skip many technical steps in our approach.

As a particular consequence, we establish the following permanence property.
\begin{Corint}\label{Corint:simple}
Simplicity of \Cs-algebras is stable under taking
reduced twisted crossed products over discrete \Cs-simple groups.
\end{Corint}

The problem on the classification of tracial weights is more subtle.
For instance, in the non-unital case the problem involves \emph{unbounded} functionals.
We introduce the following set as our replacement of the unit.
At first glance, the definition of the set may look quite technical and awkward.
However we will see in Theorem \ref{Thm:comparison} that the set has a useful natural characterization. 
Let $\Ped{A}$ denote the Pedersen ideal of $A$ (see Section \ref{SS:Ped} for details).
For $a\in A_+$, we denote by
\[\Her{a, A}:=\cl{aAa}\]
 the hereditary \Cs-subalgebra of $A$
generated by $a$.
\begin{Defint}\label{Def:cI}
For a twisted action $(\alpha, u) \colon \Gamma \curvearrowright A$ and
a proper tracial weight $\tau \in {\rm T}(A)$,
define $\cI(\alpha, \tau)$ to be the set of all $a \in (\Ped{A}_{+})_1$ satisfying
the following condition:
For any $g\in \Gamma$,
there is a bounded sequence $(v_{n})_{n=1}^{\infty}$ in $A$
satisfying
\[v_n a v_n^\ast \in \Her{\alpha_g(a), A},\]
\[\lim_n \|a^{1/2}v_n^\ast v_n a^{1/2} -a \|_\tau =0,\quad \lim_n \|v_n a v_n^\ast -\alpha_g(a) \|_\tau =0.\]
\end{Defint}

The next theorem gives a useful natural characterization of the set $\cI(\alpha, \tau)$.
\begin{Thmint}\label{Thm:comparison}
Let $\tau \in {\rm T}(A)$.
Then $a\in (\Ped{A}_{+})_1$ sits in $\cI(\alpha, \tau)$
if and only if $\nu(a^n)=\nu(\alpha_g(a^n))$
for all $n\in \IN$, $g\in \Gamma$, and $\nu \in {\rm T}(A)$ with $\nu\leq \tau$.
\end{Thmint}
Interestingly, the proof of Theorem \ref{Thm:comparison} (and hence Main Theorem \ref{MThm:trace} below) involves
von Neumann algebra theory.

We say that a subset $S \subset (A_{+})_1$ is \emph{non-degenerate}
if its strict closure in the multiplier algebra $\cM(A)$ contains the identity.
Now we are able to state our main theorem on the classification of tracial weights.
\begin{MThm}\label{MThm:trace}
Let $(\alpha, u) \colon \Gamma \curvearrowright A$ be a twisted action.
Let $\tau \in {\rm T}(A \rca{\alpha, u} \Gamma)$.
Assume that $\cI(\alpha, \tau|_A) \subset (A_+)_1$ is non-degenerate.
Then $\tau$
factors through $E_{R_a(\Gamma)}$.
\end{MThm}
This together with Theorem \ref{Thm:comparison} in particular leads to the following permanence property.
\begin{Corint}\label{Corint:trace}
Uniqueness of proper tracial weights $($up to scalar multiple$)$ is stable under taking
the reduced twisted crossed product \Cs-algebra of a trace preserving twisted action
over a discrete group with trivial amenable radical.
\end{Corint}

We again note that in the unital case (which is the case established by Bryder--Kennedy \cite{BK} in the twisted case),
one can skip many technical steps in our proof.

Our results have consequences for the reduced \Cs-algebras of \emph{locally compact groups}.
Because we are interested in non-discrete groups,
we first extend the definition of the unique trace property to
general locally compact groups.
\begin{Defint}
We say that a unimodular locally compact group $G$ has the \emph{unique trace property}
if, up to scalar multiple, the Plancherel tracial weight is the only proper tracial weight on $\rg(G)$.
\end{Defint}
We refer the reader to Section 2.6 of \cite{Rau}
for the definition and basic facts on the Plancherel weights.
For discrete groups, clearly this definition is equivalent to the usual definition.
We note that for unimodular groups $G$ with the unique trace property,
one can recover the group von Neumann algebra $L(G)$ from $\rg(G)$
by taking the strong closure in the GNS representation of any tracial weight.

By applying Corollaries \ref{Corint:simple} and \ref{Corint:trace} to the reduced group \Cs-algebras,
we establish the following stability of \Cs-simplicity
and the unique trace property
of \emph{locally compact} groups.
(Note that stronger statements are shown in Theorem 1.4 of \cite{BKKO} in the discrete group case.
See also Corollary 4.7 in \cite{BK} for a proof of the following theorem in the discrete group case.)
We stress that in both statements,
we unavoidably encounter twisted crossed products of
\emph{non-unital} \Cs-algebras (namely, the reduced group \Cs-algebras of open normal subgroups). This illustrates importance 
to generalize the results in \cite{BKKO}, \cite{BK} to the non-unital case.
\begin{MThm}\label{MThm:lcg}
Let $G$ be a locally compact group with an open normal subgroup $N$.
Then the following statements hold true.
\begin{enumerate}[\upshape(1)]
\item
If $N$ and $G/N$ are \Cs-simple, then so is $G$.
\item 
If $N$, $G$ are unimodular
and both $N$ and $G/N$ have the unique trace property,
then so does $G$.
\end{enumerate}
\end{MThm}
For the proof of statement (2),
we note that when both $N$ and $G$ are unimodular,
the conjugate action $G \curvearrowright N$ preserves the Haar measure on $N$.

The existence of non-discrete \Cs-simple groups
was a longstanding open problem explicitly asked by de la Harpe (see \cite{Har85} and \cite{Har}, Question 5).
Recently such groups were constructed in \cite{Suzsim}, \cite{Suz}\footnote{We note that the solution of the problem was originally claimed in \cite{Rau}.
However it turned out that the arguments in \cite{Rau} contain a fatal flaw.
See the erratum \cite{Rau2} for details.}.
Note that all unimodular \Cs-simple groups given there also have the unique trace property.
 \Cs-simplicity and the unique trace property of non-discrete groups
are still quite mysterious.
We believe that Main Theorem \ref{MThm:lcg}
sheds a new light on these properties.

\subsection*{Organization of the article}
In Section \ref{S:Pre}, we recall some facts used in the article.
Notations and terminologies are also fixed.
In Section \ref{S:simple} we consider the simplicity of the reduced (twisted) crossed products.
Finally, in Section \ref{S:trace}, we study proper tracial weights on the reduced (twisted) crossed products.
\section{Preliminaries}\label{S:Pre}
Here we recall some facts used in the present article.
We also fix some notations.

\subsection{Notations and conventions}
\begin{itemize}
\item
Throughout the article, let $A$ be a \Cs-algebra,
let $\Gamma$ be a discrete group, and let $\fH$ be a Hilbert space.
\item A \emph{$\Gamma$-\Cs-algebra} is a \Cs-algebra equipped with a $\Gamma$-action.
\item An \emph{ideal} of a \Cs-algebra is assumed to be self-adjoint and norm closed.
An \emph{algebraic ideal} of a \Cs-algebra is a self-adjoint ideal which is not required to be norm closed.
\item Set $(A)_1:= \{a\in A: \|a\| \leq 1\}$, $A_+:= \{ a\in A: a \geq 0\}$.
More generally, for a subset $S \subset A$,
we set $(S)_1:= S \cap (A)_1$, $S_+:=S\cap A_+$.
\item For a subset $S \subset A$,
denote by $\cl{S}$ the norm closure of $S$.
Denote by $\Cso(S)$ the (not necessary unital) \Cs-subalgebra of $A$ generated by $S$.
\item When $A$ is unital, let $A^{\mathrm u}$ denote the unitary group of $A$.
\item A series $\sum_n a_n$ in a \Cs-algebra is required to converge in norm.
\item Denote by $\cM(A)$ the multiplier algebra of $A$.
\item Denote by $\IB(\fH)$, $\IK(\fH)$, $\cU(\fH)$
the \Cs-algebra of all bounded operators, the \Cs-algebra of all compact operators,
and the group of unitary operators on $\fH$ respectively.
\end{itemize}
\subsection{Pedersen ideal}\label{SS:Ped}
Pedersen \cite{Ped} discovered that any \Cs-algebra $A$
has a (unique) smallest norm dense algebraic ideal $\Ped{A}$.
The algebraic ideal $\Ped{A}$ is referred to as the \emph{Pedersen ideal} of $A$.
It is clear from the definition of $\Ped{A}$ given in \cite{Ped} that
any \Cs-algebra inclusion $A \subset B$
restricts to the inclusion $\Ped{A} \subset \Ped{B}$.
For details, we refer the reader to Section 5.6 of \cite{Pedbook}.

\subsection{Tracial weights on \Cs-algebras}
Recall that a \emph{tracial weight} $\tau$ on a \Cs-algebra $A$
is a map
\[\tau \colon A_+ \rightarrow [0, \infty]\]
satisfying
\[\tau(a+ b)=\tau(a)+\tau(b),\quad \tau(\lambda a) = \lambda \tau(a), \quad
\tau(x^\ast x)= \tau(x x^\ast)\]
for all $a, b \in A_+$, $\lambda \in [0, \infty)$, $x\in A$.
Here we employ the conventions $\lambda+\infty=\infty+\lambda=\infty+\infty=\infty$ and $0 \cdot \infty =0$.
A tracial weight $\tau$ is said to be \emph{proper}
if it is lower semi-continuous, densely defined, and nonzero.
We denote by ${\rm T}(A)$ the set of all proper tracial weights on $A$.
It is known that each $\tau \in {\rm T}(A)$
gives rise to a tracial positive linear functional on $\Ped{A}$
by linear extension of the restriction $\tau|_{\Ped{A}_+}$. 
We use the same symbol $\tau$ to denote this linear functional.
Moreover this correspondence is bijective; see Proposition 5.6.7 of \cite{Pedbook}.
We note that when $\tau\in {\rm T}(A)$ is bounded, that is, when $\infty\not\in \tau(A_+)$,
then $\tau$ extends to a tracial positive (bounded) linear functional on $A$.
We also denote this extension by $\tau$.

For $\tau, \nu \in {\rm T}(A)$,
we write $\nu \leq \tau$ when they satisfy the inequality
$\nu(a) \leq \tau(a)$ for all $a\in A_+$.
For $\tau \in {\rm T}(A)$ and $a \in \Ped{A}$,
set
\[\|a\|_\tau := \tau(a^\ast a)^{1/2}.\]
This defines a seminorm on $\Ped{A}$.
By the tracial condition, for any $a\in \Ped{A}$ and any $b\in A$, one has
\[\|ab\|_\tau\leq \|a\|_\tau \|b\|.\]

In the study of proper tracial weights, the following fact is useful.

\begin{Lem}[\cite{Pedbook}, Proposition 5.6.2] \label{Lem:her}
For any $a\in \Ped{A}_+$,
one has $\Her{a, A} \subset \Ped{A}$.
Consequently every $\tau \in {\rm T}(A)$
is bounded on $\Her{a, A}$.
\end{Lem}

\subsection{Twisted actions and Packer--Raeburn theorem}
Let $G$ be a locally compact group.
A \emph{twisted action} (in the sense of Busby--Smith\footnote{Our results (Main Theorems \ref{MThm:simple} and \ref{MThm:trace}) are also valid for the reduced twisted products of twisted actions in Green's sense \cite{Gre}. 
This is because analogues of the statements that we used
to reduce to the case of genuine actions also hold true in Green's setting.
See e.g., \cite{Ech} for details.
We employ the Busby--Smith formulation \cite{BS} just for compatibility with the setting in \cite{BK}.})
$(\alpha, u) \colon G \rightarrow A$
is a pair of measurable maps
\[\alpha \colon G \rightarrow \mathrm{ Aut}(A),\quad
u \colon G \times G \rightarrow \cM(A)^{\rm u}\]
satisfying the cocycle relation
\[\alpha_s \circ \alpha_t = \ad(u(s, t)) \circ \alpha_{st},\quad \alpha_r(u(s, t)) u(r, st)= u(r, s) u(rs, t), \quad u(e, s)= u(s, e)=1\]
for all $s, t, r\in G$ (see e.g., Definition 2.1 of \cite{PR} for details).
Analogous to usual actions
(which correspond to the case $u=1$),
each twisted action $(\alpha, u)\colon G \curvearrowright A$ associates a canonical \Cs-algebra $A \rca{\alpha, u} G$,
called the \emph{ reduced twisted crossed product \Cs-algebra} of $(a, u)$.
For a discrete group $\Gamma$, the reduced twisted crossed product
of $(\alpha, u) \colon \Gamma \curvearrowright A$ can be defined as follows.
First take a faithful non-degenerate $\ast$-representation $\pi \colon A \rightarrow \IB(\fH)$.
Define a $\ast$-representation
$\tilde{\pi}_\alpha\colon A \rightarrow \IB(\fH \otimes \ell^2\Gamma)$ and a map $\lambda^u \colon \Gamma \rightarrow \cU(\fH \otimes \ell^2 \Gamma)$ by the formulas
\[\tilde{\pi}_\alpha(a)(\xi \otimes \delta_t) := \pi(\alpha_{t^{-1}}(a))(\xi) \otimes \delta_t,\]
\[\lambda^u_g(\xi \otimes \delta_t):=\pi(u(t^{-1}g^{-1}, g))(\xi) \otimes \delta_{gt}.\]
Then define $A \rca{\alpha, u} \Gamma := \Cso(\tilde{\pi}(A)\cdot \lambda^u(\Gamma))$.
Note that up to canonical isomorphism, the \Cs-algebra $A \rca{\alpha, u}\Gamma$
is independent of the choice of $\pi$.
(Alternatively one can define $A \rca{\alpha, u}\Gamma$ 
as a \Cs-completion of the \emph{twisted algebraic crossed product} $A \rtimes_{\mathrm{alg}, \alpha, u}\Gamma$ without choosing a particular $\pi$.)
Note also that $\lambda^u_g \in \cM(A \rca{\alpha, u} \Gamma)^{\mathrm u}$.

Importance of the twisted actions comes from group extensions.
Let $G$ be a locally compact group given by the (topological) group extension
\[1\rightarrow N \rightarrow G \rightarrow Q \rightarrow 1\]
of locally compact groups with a measurable cross section $Q \rightarrow G$
(which always exists when $Q$ is discrete or $G$ is second countable.)
Then, for any twisted action $(\alpha, u) \colon G \curvearrowright A$,
one can decompose $A\rca{\alpha, u} G$
into the iterated reduced twisted crossed product
$(A \rca{\alpha|_N, u|_N} N) \rca{\beta, v} Q$
for an associated twisted action
$(\beta, v) \colon Q \curvearrowright A \rca{\alpha|_N, u|_N}N$.
Here the appearance of the $2$-cocycle $v$ is essential even when the original $2$-cocycle $u$ is trivial.
For a proof, see Theorem 4.1 of \cite{PR} for instance.
(Note that the proof there is given for the \emph{full} twisted crossed product \Cs-algebras.
However at least when $Q$ is discrete, it is not hard to see that the $\ast$-isomorphism given there
passes to the $\ast$-isomorphism on the reduced \Cs-algebras.
Cf.~ Proposition 8.5 of \cite{Ech} and the comment below it.)

Let
$(\alpha, u) \colon \Gamma \curvearrowright A$ be a twisted action.
Then, analogous to the usual reduced crossed product,
for any subgroup $\Lambda$ of $\Gamma$,
we have a canonical inclusion
\[A \rca{\alpha|_\Lambda, u|_\Lambda} \Lambda \subset A \rca{\alpha, u} \Gamma.\]
(To see this, observe that for each $s\in \Gamma$,
the restrictions of the maps $\tilde{\pi}_\alpha$, $\lambda^u |_\Lambda$ to $\fH \otimes \ell^2(\Lambda s)$ is conjugate to the maps $(\pi\circ \alpha_{s^{-1}})^{\sim}_{\alpha|_\Lambda}$, $\lambda^{u|_\Lambda}$ respectively
by the unitary operator $V \colon \fH \otimes \ell^2 (\Lambda s) \rightarrow \fH \otimes \ell^2 \Lambda$ defined by
$V(\xi \otimes \delta_{ts}):=\pi(u(s^{-1}, t^{-1}))(\xi) \otimes \delta_t$.
This shows that the map $a \lambda^{u|_\Lambda}_g\mapsto a \lambda^u_g$; $a\in A$, $g\in \Lambda$,
extends to an injective $\ast$-homomorphism $A \rca{\alpha|_\Lambda, u|_\Lambda} \Lambda \rightarrow A \rca{\alpha, u} \Gamma$.)
Moreover the inclusion admits a (faithful) canonical conditional expectation
\[E_\Lambda \colon A \rca{\alpha, u} \Gamma \rightarrow A \rca{\alpha|_\Lambda, u|_\Lambda} \Lambda\]
satisfying $E_\Lambda(A \cdot \lambda_g^u)=0$ for all $g\in \Gamma \setminus \Lambda$.
We refer to $E_\Lambda$ as the \emph{subgroup conditional expectation}.
(To see such an $E_\Lambda$ really exists, 
represent $A \rca{\alpha, u} \Gamma$ on $\fH \otimes \ell^2 \Gamma$
as in the definition.
Denote by
$p \colon \fH \otimes \ell^2 \Gamma \rightarrow \fH \otimes \ell^2 \Lambda$
 the orthogonal projection.
Then the map \[\IB(\fH \otimes \ell^2 \Gamma)\ni x\mapsto p x p\in \IB(\fH \otimes \ell^2 \Lambda)\]
restricts to the desired conditional expectation.)

For $\tau \in {\rm T}(A \rca{\alpha, u} \Gamma)$ and a normal subgroup $\Lambda \subset \Gamma$,
we say $\tau$ \emph{factors through} $E_\Lambda$
if it satisfies
$\tau=\tau \circ E_\Lambda$.
Note that for any $\tau \in {\rm T}(A \rca{\alpha|_\Lambda, u|_\Lambda} \Lambda)$
which is invariant under conjugation of $\lambda^u_g$; $g\in \Gamma$,
one has
$\tau \circ E_\Lambda \in {\rm T}(A \rca{\alpha, u} \Gamma)$.

\begin{Def}\label{Def:ee}
Two twisted actions
$(\alpha, u), (\beta, v) \colon \Gamma \curvearrowright A$
are said to be \emph{exterior equivalent}
if there is a
map $w\colon \Gamma \rightarrow \cM(A)^{\mathrm u}$
satisfying
\[\beta_s=\ad(w_s)\circ \alpha_s,\quad v(s, t)=w_s\alpha_s(w_t)u(s, t) w_{st}^\ast.\] 
\end{Def}
It is known that exterior equivalence leads to
a natural $\ast$-isomorphism between the reduced twisted crossed product \Cs-algebras.
Indeed, keep the setting from Definition \ref{Def:ee}
and realize both $A \rca{\alpha, u} \Gamma$ and $ A \rca{\beta, v} \Gamma$ on $\fH \otimes \ell^2\Gamma$ as in the definition by using the same representation $\pi\colon A \rightarrow \IB(\fH)$.
Define $U \in \cU(\fH \otimes \ell^2\Gamma)$
to be
\[U(\xi \otimes \delta_t):=\pi(w_{t^{-1}}) (\xi) \otimes \delta_t.\]
Then we have
\[U(A \rca{\alpha, u} \Gamma) U^\ast = A \rca{\beta, v} \Gamma.\]
Indeed direct calculations show
\[\ad(U) \circ \tilde{\pi}_\alpha=\tilde{\pi}_\beta,\quad
U\tilde{\pi}_\alpha(w_g) \lambda^u_gU^\ast = \lambda^v_g.\]
This fact is observed in Corollary 3.7 in \cite{PR} for the full twisted crossed product \Cs-algebras.
It is clear from these two equalities that the above $\ast$-isomorphism commutes with all subgroup conditional expectations.

The statements of the present article cover the reduced twisted crossed product \Cs-algebras.
However the proofs can be reduced to the untwisted case.
This is due to the following theorem of Packer--Raeburn \cite{PR}.
We emphasis that the theorem involves non-unital \Cs-algebras
even if one is only interested in unital \Cs-algebras.
This is one of the main reasons why we believe it is natural and important to extend the theorems in \cite{BKKO} to the non-unital case.
\begin{Thm}[\cite{PR}, Theorem 3.4]\label{Thm:PR}
Let $(\alpha, u) \colon \Gamma \curvearrowright A$ be a twisted action.
Then the stabilization
$(\alpha \otimes \id_{\IK(\ell^2\Gamma)}, u \otimes \id_{\IK(\ell^2\Gamma)})\colon \Gamma \curvearrowright A \otimes \IK(\ell^2\Gamma)$
is exterior equivalent to a genuine action
$\beta \colon \Gamma \curvearrowright A \otimes \IK(\ell^2\Gamma)$.

Thus there is a $\ast$-isomorphism
\[ (A \rca{\alpha, u} \Gamma) \otimes \IK(\ell^2 \Gamma) \cong (A \otimes \IK(\ell^2 \Gamma)) \rca{\beta} \Gamma\]
which commutes with all subgroup conditional expectations.
\end{Thm}

\begin{Rem}\label{Rem:el}
It is easy to see that
the sets $\cI(\alpha, \tau)$ in Definition \ref{Def:cI}
are unchanged by exterior equivalence.
Thus, for a twisted action $(\alpha, u) \colon \Gamma \curvearrowright A$,
let $\beta \colon \Gamma \curvearrowright A \otimes \IK(\ell^2 \Gamma)$
be the action given in Theorem \ref{Thm:PR}.
Then for any $\tau\in {\rm T}(A)$, we have
\[\cI(\alpha, \tau) \cdot (\Ped{\IK(\ell^2\Gamma)}_+)_1 \subset \cI(\beta, \tau \otimes \mathrm{Tr}).\]
Here $\mathrm{Tr}$ denotes the canonical (proper) tracial weight on $\IK(\ell^2 \Gamma)$.
In particular, when $A$ is unital or more generally when
$\cI(\alpha, \tau)$ is non-degenerate, the set $\cI(\beta, \tau \otimes \mathrm{Tr})$ is non-degenerate.
 \end{Rem}
\subsection{Boundary actions, \Cs-simplicity, and amenable radical}
Recall that a locally compact group $G$ is \emph{\Cs-simple}
if its reduced group \Cs-algebra $\rg(G)$ is simple.

Recall that a (continuous) action $G\curvearrowright X$ on a (non-empty) compact space
is said to be a \emph{boundary action}
if any $G$-invariant weak-$\ast$ closed non-empty subset of the state space of $C(X)$ contains all characters on $C(X)$.
A (topological) \emph{$G$-boundary} is a compact space $X$ equipped with a boundary action $G \curvearrowright X$.

Kalantar--Kennedy \cite{KK} discovered the following unexpected deep connection between \Cs-simplicity and boundary actions of discrete groups.
\begin{Thm}[see \cite{KK}, Theorem 1.5]
A discrete group $\Gamma$ is \Cs-simple if and only if
there is a $($topologically$)$ free $\Gamma$-boundary.
\end{Thm}

The article \cite{BKKO} establishes another deep connection
of boundary actions with the unique trace property of
discrete groups.

Recall that for a discrete group $\Gamma$,
the \emph{amenable radical} $R_a(\Gamma)$ of $\Gamma$
is the largest amenable normal subgroup in $\Gamma$.
\begin{Thm}[\cite{BKKO}, Theorem 1.3, Proposition 2.8, \cite{Fur}]\label{Thm:BKKO}
For a discrete group $\Gamma$,
the following conditions are equivalent.
\begin{enumerate}[\upshape(1)]
\item The group $\Gamma$ has the unique trace property.
\item The group $\Gamma$ admits a faithful $\Gamma$-boundary.
\item The amenable radical $R_a(\Gamma)$ is trivial.
\end{enumerate}
\end{Thm}
We also refer the reader to \cite{Haa} and \cite{Ken}
for alternative approaches to these theorems and further characterizations of
\Cs-simplicity and the unique trace property for discrete groups.

The above two theorems show that \Cs-simplicity implies the unique trace property for discrete groups.
The converse is not true as counterexamples are constructed by Le Boudec \cite{Bou}.

Among boundary actions of a discrete group $\Gamma$, there is a universal action.
The universal boundary action $\beta \colon \Gamma \curvearrowright \partial_F \Gamma$
is referred to as the \emph{Furstenberg boundary} of $\Gamma$.
Furman \cite{Fur} has shown the equality
\[\ker(\beta)=R_a(\Gamma).\]
(See also \cite{BKKO}, Proposition 2.8 for a short proof.)
The Furstenberg boundary has the following remarkable property,
which is a key ingredient of the simplicity theorem.
\begin{Thm}[\cite{KK}, Theorem 3.12, \cite{Ham85}]\label{Thm:Fur}
Any unital $\Gamma$-\Cs-algebra $C$ admits a $\Gamma$-equivariant unital completely positive map
\[C \rightarrow C(\partial_F \Gamma).\]
Moreover in the case of $C=C(X)$ where $X$ is a $\Gamma$-boundary,
such a map is unique and is an injective $\ast$-homomorphism.
\end{Thm}

\section{Simplicity of reduced twisted crossed products}\label{S:simple}
In this section, we give a proof of Main Theorem \ref{MThm:simple}.
The next lemma follows from the proof of Lemma A.2 of \cite{SuzMA}.
\begin{Lem}\label{Lem:cf}
Let $a\in (A_+)_1$ be a completely full element.
Then there is a net $((v_{j, n})_{n=1}^{\infty})_{j\in J}$ of sequences in $A$
satisfying the following conditions.
\[\sum_n v_{j, n}v_{j, n}^\ast \leq 1,\quad
\stlim\lim_{j\in J}\left(\sum_n v_{j, n}a v_{j, n}^\ast\right) = 1.\]
\end{Lem}

For the simplicity theorem,
the next lemma is the crucial step.
\begin{Lem}[see Lemma 7.2 of \cite{BKKO} for the unital case]
Let $A$ be a $\Gamma$-\Cs-algebra with a completely full element.
Let $I$ be a proper ideal of $A \rc \Gamma$.
Let $X$ be a $\Gamma$-boundary.
Then $I$ generates a proper ideal in $(A \otimes C(X)) \rc \Gamma$.
\end{Lem}
\begin{proof}
Fix a non-degenerate $\ast$-representation
$\pi \colon A \rc \Gamma \rightarrow \IB(\fH)$ with
$\ker(\pi)=I$.
By the Arveson extension theorem,
one can take a completely positive extension $\tilde{\pi} \colon (A \otimes C(X)) \rc \Gamma \rightarrow \IB(\fH)$ of 
$\pi$.
Note that $A \rc \Gamma$ sits in the multiplicative domain of $\tilde{\pi}$.
We use the same symbol $\tilde{\pi}$ to denote the strictly continuous extension of $\tilde{\pi}$ on $\cM((A \otimes C(X)) \rc \Gamma)$.
(See Corollary 5.7 of \cite{Lan} for the existence of such an extension.)
Let $C$ be the \Cs-algebra on $\fH$ generated by $\tilde{\pi}(C(X))$.
Note that $C$ contains $\id_{\fH}$.
We equip $C$ with the $\Gamma$-action by conjugation.
Note that $\tilde{\pi}|_{C(X)} \colon C(X) \rightarrow C$ is $\Gamma$-equivariant.
By Theorem \ref{Thm:Fur}, one has a $\Gamma$-equivariant unital completely positive map $\psi\colon C \rightarrow C(\partial_F \Gamma)$.
Moreover, since $X$ is a $\Gamma$-boundary,
the composite $\psi \circ \tilde{\pi}|_{C(X)}$ is an (injective) $\ast$-homomorphism.
This shows that $\tilde{\pi}(C(X)^{\mathrm u})$ sits in the multiplicative domain of $\psi$.
Hence $\psi$ is in fact a $\ast$-homomorphism.
Put $K:= \ker(\psi) \subset C$.

Now set $D:=\Cso(\tilde{\pi}((A \otimes C(X)) \rc \Gamma))$,
$L:= \cl{K\cdot D}$.
Here and below we regard $C \subset \cM(D)$ in the obvious way. 
Note that $L$ is an ideal of $D$, because $K$ is a $\Gamma$-invariant ideal of $C$.
We will show that $L \neq D$.
Since the composite of $\tilde{\pi}$ with the quotient map $D \rightarrow D/L$
gives a non-degenerate $\ast$-homomorphism on $(A \otimes C(X)) \rc \Gamma$
which vanishes on $I$,
this completes the proof.

To show the claim, take a completely full element $a\in (A_{+})_1$.
We will show that the map
\[j_a\colon C\ni c \mapsto \pi(a) c\in D\]
 is isometric.
To see this, we apply Lemma \ref{Lem:cf} to obtain a net $((v_{j, n})_{n=1}^\infty)_j$ of sequences in $A$
satisfying \[\sum_{n} v_{j, n} v_{j, n}^\ast \leq 1,\quad
\stlim\lim_j \sum_{n} v_{j, n} a v_{j, n}^\ast = 1.\]
Since $\pi(A)$ commutes with $C$, for any $c\in C$ and $j$,
we have
\[\|\pi(\sum_n v_{j, n} a v_{j, n}^\ast) c\|=\|\sum_n \pi(v_{j, n}) \pi( a) c \pi(v_{j, n}^\ast)\|\leq \|\pi(a)c\|.\]
Since $\pi$ is non-degenerate, the left most quantity converges to $\|c\|$ as $j \rightarrow \infty$.
Hence
\[\|c\| \leq \|\pi(a)c\|.\]
The reverse inequality is clear.

Choose an approximate unit $(e_i)_i$ of $K$.
Observe that $x\in D$ sits in $L$
if and only if $\lim_i (e_i x) = x$ in norm.
Hence for any $c\in C$ with $j_a(c)\in L$,
one has
\[\lim_i j_a((1-e_i)c)=\lim_i \left( (1-e_i)j_a(c)\right)=0.\]
As $j_a$ is isometric, this yields $\lim_i e_i c =c$ and therefore $c\in K$.
This proves $j_a^{-1}(L) \subset K \subsetneq C$
and hence $L \neq D$.
\end{proof}

Since Lemma 7.2 of \cite{BKKO} is the only point
where the unit element is used,
the rest of the proof (i.e., Lemma 7.3 of \cite{BKKO})
works verbatim.
Therefore we complete the proof of Main Theorem \ref{MThm:simple}.
\section{Classification of tracial weights on reduced crossed products}\label{S:trace}
In this section we prove Main Theorem \ref{MThm:trace}.

We first prove Theorem \ref{Thm:comparison},
which gives a useful characterization of the set $\cI(\alpha, \tau)$ in Definition \ref{Def:cI}.
It is a direct consequence of the following general criterion.
\begin{Lem}\label{Lem:comparison}
Let $\tau \in {\rm T}(A)$ and $a, b \in \Ped{A}_+$.
Assume that for any $\nu \in {\rm T}(A)$ with $\nu \leq \tau$
and any $n\in \IN$, one has
$\nu(a^n)=\nu(b^n)$.
Then there is a sequence
$(v_n)_n$ in $(A)_1$
satisfying
\[v_n v_n^\ast \in \Her{b, A},\]
\[\lim_n \|a- a^{1/2}v_n^\ast v_n a^{1/2} \|_\tau = 0,\quad \lim_n \|b- v_n a v_n^\ast \|_\tau =0.\]
Conversely, assume such a bounded sequence $(v_n)_n$ exists.
Then $\nu(a^n)= \nu(b^n)$ for all $n\in \IN$
and all $\nu \in {\rm T}(A)$ with $\nu \leq \tau$.
\end{Lem}
\begin{proof}
By passing to a quotient of $A$, we may assume that $\tau$ is faithful.
(Indeed for any quotient $\ast$-homomorphism $p \colon A \rightarrow A/J$,
$p((A)_1)$ is norm dense in $(A/J)_1$.)

We first remark that $\tau$ is bounded on $\Her{a+b, A}$ by Lemma \ref{Lem:her}.
By multiplying with appropriate positive scalars if necessary, throughout the proof, we may assume that
\[\|a\|, \|b\|, \|\tau|_{\Her{a+b, A}}\| \leq 1.\]

Assume that the elements $a, b\in \Ped{A}_+$ satisfy the first condition.
Let $\pi_\tau$ denote the GNS representation of the weight $\tau$ (see Proposition 5.1.3 of \cite{Pedbook}).
We regard $A \subset M:=\pi_\tau(A)''$ in the obvious way.
By Lemma \ref{Lem:her} and the assumption, we have
\[\nu(\chi_{[s, t)}(a))=\nu(\chi_{[s, t)}(b))< \infty\]
for all $0< s < t$ and $\nu\in {\rm T}(A)$ with $\nu \leq \tau$.
Note that any $\nu \in {\rm T}(A)$ with $\nu \leq \tau$ extends to a normal tracial weight on $M$ which we denote by the same symbol $\nu$.
By the comparison theorem,
this implies the Murray--von~Neumann equivalence of 
the two spectral projections $\chi_{[s, t)}(a)$, $\chi_{[s, t)}(b)$ in $M$.
By considering appropriate (norm) approximations of $a$ and $b$ by
(orthogonal) linear combinations of the above spectral projections, one can find a sequence $(u_n)_n$ in $M^{\rm u}$
satisfying
\[\lim_n \|u_n a u_n^\ast- b\|=0.\]
Then by the Kaplansky density theorem, one can choose a net $(x_i)_i$ in $(A)_1$
satisfying
$x_i a x_i^\ast \rightarrow b$ in the strong$\ast$ topology of $M$.

Let $\varepsilon>0$.
Fix $s>0$
satisfying $\| b - b^{1+ 2s}\|< \varepsilon$.
Observe that
$b^{s}x_i a x_i^\ast b^{s} \rightarrow b^{1+2s}$ in the strong$\ast$ topology of $M$.
Since $\tau$ is bounded on $\Her{b, A}$, for sufficiently large $i$, the element
\[v:= b^{s} x_i\in (A)_1\]
satisfies
\[\|v a v^\ast - b\|_\tau\leq \|b^s (x_ia x_i^\ast-b)b^s\|_\tau + \|b-b^{1+2s}\| < \varepsilon.\]
(Here we again recall the assumptions $\|b\|, \|\tau|_{\Her{a+b, A}}\|\leq 1$, which implies the inequality $\|b-b^{1+2s}\|_\tau \leq \|b-b^{1+2s}\|$.)
Clearly $v v^\ast \in \Her{b, A}$.
Since $\tau$ is contractive on $\Her{a+b, A}$, the Cauchy--Schwarz inequality yields
\[|\tau(a^{1/2}v^\ast v a^{1/2}) - \tau(a)| =|\tau(v a v^\ast -b)| \leq \|vav^\ast -b \|_\tau <\varepsilon.\] 
Then, as $a- a^{1/2}v^\ast v a^{1/2}$ is a positive contractive element,
we conclude
\[\|a^{1/2}v^\ast v a^{1/2} -a\|_\tau^2= \tau((a- a^{1/2}v^\ast v a^{1/2})^2) \leq \tau(a- a^{1/2}v^\ast v a^{1/2})< \varepsilon.\]
Since $\varepsilon>0$ was arbitrary, this proves that the elements $a$ and $b$ satisfy the second condition in the statement.

Conversely, assume that the elements $a, b \in \Ped{A}_+$ satisfy the second condition in the statement.
Choose a bounded sequence $(v_n)_n \subset A$ which witnesses the condition.
We first observe that, for any $m\in \IN$,
\begin{align*}\|a^m - a^{m/2} v_n^\ast v_n a^{m/2}\|_\tau&= \|a^{(m-1)/2}(a - a^{1/2}v_n^\ast v_n a^{1/2}) a^{(m-1)/2}\|_\tau\\
&\leq \|a^{(m-1)/2}\| \|a - a^{1/2}v_n^\ast v_n a^{1/2}\|_\tau \|a^{(m-1)/2}\| .\end{align*}
This proves $\lim_n \|a^m - a^{m/2} v_n^\ast v_n a^{m/2}\|_\tau=0$ for $m\in \IN$. (The case $m=1$ is the assumption.)
We next show the equality
 \[\lim_n \|b^m- v_n a^{m} v_n^\ast\|_\tau =0\]
 for $m \in \IN$ by induction. 
The case $m=1$ is a part of the present assumption.
 Assume that the equality holds true for a given $m\in \IN$.
 It suffices to show the equality for $m+1$.
Note that for any $n\in \IN$, we have
 
 \begin{align*}
 \|b^{m+1}-v_n a^{m+1} v_n^\ast\|_\tau &\leq \|b(b^m-v_n a^m v_n^\ast)\|_\tau + \|bv_n a^m v_n^\ast - v_n a^{m+1} v_n^\ast\|_\tau\\
&\leq \|b\|\|b^m-v_n a^m v_n^\ast\|_\tau + \|(b-v_n a v_n^\ast)v_n a^m v_n^\ast\|_\tau + \|v_n a (v_n^\ast v_n -1) a^m v_n^\ast\|_\tau\\
&\leq \|b\|\|b^m-v_n a^m v_n^\ast\|_\tau + \|(b-v_n a v_n^\ast)\|_\tau \|v_n a^m v_n^\ast\|\\
&\quad + \|v_na^{1/2} \| \|a^{1/2} (v_n^\ast v_n -1)a^{1/2}\|_\tau \| a^{m-1/2} v_n^\ast\|
 \end{align*}
 As the last quantity converges to $0$ as $n\rightarrow \infty$, this proves the equality for $m+1$.
 
We note that the same convergence condition holds true for all seminorms $\|\cdot \|_\nu$ with $\nu \leq \tau$ in place of $\|\cdot \|_\tau$,
because $\|\cdot \|_\nu \leq \|\cdot \|_\tau$. 
Now the Cauchy--Schwarz inequality implies
\[\nu(a^m)=\lim_n \nu(a^{m/2} v_n^\ast v_n a^{m/2}), \quad \nu(b^m)=\lim_n \nu(v_n a^{m} v_n^\ast)\]
for all $\nu\in {\rm T}(A)$ with $\nu \leq \tau$ and $m\in \IN$, whence the desired condition $\nu(a^m)=\nu(b^m)$ follows from the tracial condition.
\end{proof}

The next lemma gives a non-unital version of the multiplicative domain argument.
For completeness of the article, we include a proof.
\begin{Lem}\label{Lem:bilinear}
Let $a\in A_{+}$. Let $C$ be a unital $\Gamma$-\Cs-algebra.
Let $\varphi \in \Her{a, (A \otimes C) \rc \Gamma}^\ast$
be positive.
Assume that there is a character $\omega$ on $C$
with $\varphi(a \otimes c)=\varphi(a)\omega(c)$ for all $c\in C$.
Then $\varphi$ satisfies
\[\varphi(c_1 b c_2)= \omega(c_1)\varphi(b)\omega(c_2)\]
for all $c_1, c_2 \in C$ and $b \in \Her{a, (A \otimes C) \rc \Gamma}$.
\end{Lem}
\begin{proof}
Define a semi-inner product $\ip{\cdot, \cdot}_\varphi$ on $\cl{[(A \otimes C) \rc \Gamma] a^{1/2}}$
to be \[\ip{b_1, b_2}_\varphi:= \varphi(b_1^\ast b_2).\]
Since $\varphi$ is bounded and self-adjoint,
it suffices to show the equality $\varphi(cb)=0$ for all
$b \in a[(A\otimes C) \rc \Gamma] a$
and $c\in \ker(\omega)$.
Write $b= a^{1/2} b_1 a^{1/2} $; $b_1\in (A\otimes C) \rc \Gamma$.
Then one has $cb= (a^{1/2} \otimes c^\ast)^\ast(b_1a^{1/2})$.
By the Cauchy--Schwarz inequality, we have
\[|\varphi(cb)|= |\ip{a^{1/2}\otimes c^\ast, b_1 a^{1/2}}_\varphi| \leq \varphi(a \otimes c c^\ast )^{1/2} \varphi(a^{1/2} b_1^\ast b_1 a^{1/2})^{1/2} =0.\]
This shows the claim.
\end{proof}

The next lemma is crucial in our classification theorem of tracial weights.
In the lemma, the set $\cI(\alpha, \tau)$ (see Definition \ref{Def:cI} and Theorem \ref{Thm:comparison}) plays a prominent role.
\begin{Lem}\label{Lem:extension}
Let $\alpha \colon \Gamma \curvearrowright A$
and $\gamma \colon \Gamma \curvearrowright C$
be actions where $C$ is a unital \Cs-algebra.
For $\tau \in {\rm T}(A \rc \Gamma)$ and $a\in \cI(\alpha, \tau|_A)$,
set
\[\cE(\tau, a):=\left\{\varphi \in \Her{a, (A \otimes C) \rc \Gamma}^\ast: \varphi {\rm ~is~ a ~positive~extension ~of~} \tau|_{\Her{a, A \rc \Gamma}}\right\}.\]
Define $\iota_a \colon C \rightarrow (A\otimes C) \rc \Gamma$
by $\iota_a(c):=a \otimes c$.
Then the set
\[S:=\{\varphi \circ \iota_a\in C^\ast: \varphi \in \cE(\tau, a)\}\]
is non-empty, $\Gamma$-invariant, and weak-$\ast$ compact.

Hence in the case of $C=C(X)$ where $X$ is a $\Gamma$-boundary, for any character $\omega$ on $C(X)$,
there is $\varphi \in \cE(\tau, a)$ satisfying
\[\varphi \circ \iota_a = \tau(a) \cdot \omega.\]
\end{Lem}
\begin{proof}
By the Hahn--Banach theorem (and Lemma \ref{Lem:her}), $\cE(\tau, a)$ is non-empty.
It is clear that $\cE(\tau, a)$ and hence $S$ is weak-$\ast$ compact.

To show the $\Gamma$-invariance of $S$, take any $\varphi \in \cE(\tau, a)$ and $g\in \Gamma$.
Choose a bounded sequence $(v_{n})_{n=1}^{\infty}$ in $A$ as in
the definition of $\cI(\alpha, \tau|_A)$ for the present $a$ and $g$.

For each $n\in \IN$, define $\varphi_n \in \Her{a, (A \otimes C) \rc \Gamma}^\ast$ to be 
\[\varphi_n(x):= \varphi(\lambda_g^\ast v_{n}x v_{n}^\ast \lambda_g ).\]
Observe that each $\varphi_n$ is positive.
Moreover the sequence $(\varphi_n)_n$ is bounded.
Choose a weak-$\ast$ cluster point $\psi$ of $(\varphi_n)_n$.
Clearly $\psi$ is positive.

For any $x =a y a$ with $y\in A \rc \Gamma$, 
we have
\[\varphi_n(x)
= \tau(\lambda_g^\ast v_{n}aya v_{n}^\ast \lambda_g )
=\tau((a^{1/2} y a^{1/2})(a^{1/2} v_{n}^\ast v_{n}a^{1/2})),\]
which converges to $\tau(a^{1/2}ya^{3/2})=\tau(x)$ by the Cauchy--Schwarz inequality.
Therefore $\psi \in \cE(\tau, a)$.
Also for any $c\in C$ and any $n\in \IN$,
we have
\[\varphi_n(a \otimes c)= \varphi(\lambda_g^\ast( v_n a v_n^\ast \otimes c) \lambda_g)
= \varphi(\alpha_g^{-1}(v_n a v_n^\ast)\otimes \gamma_g^{-1}(c)).\]
Let $\tilde{\varphi} \in \cM((A\otimes C) \rc \Gamma)^\ast$
be a positive extension of $\varphi$.
Then the Cauchy--Schwarz inequality yields
\[\begin{split}|\varphi(a\otimes \gamma^{-1}_g(c)) -\varphi_n(a \otimes c)|
&=|\varphi([a- \alpha_g^{-1}(v_n a v_n^\ast)] \otimes \gamma_g^{-1}(c))|\\
&\leq \|a- \alpha_g^{-1}(v_n a v_n^\ast)\|_\tau \cdot \tilde{\varphi} (\gamma_g^{-1}(c^\ast c))^{1/2}.
\end{split}\]
Since 
\begin{align*}
\|a- \alpha_g^{-1}(v_n a v_n^\ast)\|_\tau &= \|\lambda_g^\ast (\alpha_g(a)-v_n a v_n^\ast)\lambda_g\|_\tau \\
&= \|\alpha_g(a)-v_n a v_n^\ast\|_\tau \rightarrow 0 \quad {\rm ~as~}n \rightarrow \infty,\end{align*}
we conclude
\[\psi \circ \iota_a = \gamma_g^\ast(\varphi \circ \iota_a).\]
This proves the $\Gamma$-invariance of $S$.
\end{proof}
We complete the proof of Main Theorem \ref{MThm:trace}.
\begin{proof}[Proof of Main Theorem \ref{MThm:trace}]
By Theorem \ref{Thm:PR} and Remark \ref{Rem:el}, it suffices to show the statement for genuine actions.
Since $\cI(\alpha, \tau|_A)$ is non-degenerate, it suffices to show the equality
\[\tau(a x \lambda_g a)=0\]
for all $g\in \Gamma \setminus R_a(\Gamma)$, $a\in \cI(\alpha, \tau|_A)$, $x\in A$.

Let $g, a, x$ be as above.
By \cite{Fur} (see also Proposition 2.8 in \cite{BKKO}), one can choose a character $\omega$ on $C(\partial_F \Gamma)$
and $f\in C(\partial_F \Gamma)$
satisfying $\omega(f)=1$, $\omega(\beta_g(f))=0$.
By Lemma \ref{Lem:extension}, one has 
$\varphi\in \cE(\tau, a)$ with
$\varphi \circ \iota_{a} = \tau(a) \cdot \omega$.
Then Lemma \ref{Lem:bilinear} yields
\[\begin{split}\tau(a x\lambda_g a)&= \varphi(a x \lambda_g a)\omega(f)
= \varphi(a x \lambda_g a f)\\
& = \varphi(\beta_g(f) a x \lambda_g a)=\omega(\beta_g(f))\varphi(a x \lambda_g a)\\
&=0.\end{split}\]
\end{proof}

\begin{Rem}
If one employs the Cuntz--Pedersen type equivalence relation \cite{CP}
instead of the Murray--von Neumann type relation, then one has the set
$\cJ(\alpha, \tau)$, which is (possibly) larger than $\cI(\alpha, \tau)$, consisting of all elements
$a\in (\Ped{A}_+)_1$
satisfying the following condition:
For any $g\in \Gamma$, there is a double sequence $(v_{n, m})_{n, m}$ in $A$
satisfying
\[\sum_m v_{n, m} v_{n, m}^\ast \in \Her{\alpha_g(a), A},\quad \sup_{n} \| \sum_m v_{n,m} v_{n, m}^\ast\| < \infty,\]
\[\lim_n \|\sum_m a^{1/2}v_{n, m}^\ast v_{n, m} a^{1/2} -a \|_\tau =0,\quad \lim_n \|\sum_m v_{n, m} a v_{n, m}^\ast -\alpha_g(a) \|_\tau =0.\]
Although we are not aware of any applications at the moment,
it would be worth to point out that our results in this article work without major corrections
after replacing $\cI(\alpha, \tau)$ by $\cJ(\alpha, \tau)$.
\end{Rem}
\begin{Rem}
Since any boundary action of $\Gamma$ factors through the quotient group $\Gamma/ R_a(\Gamma)$ \cite{Fur},
one can drop the condition on elements $g$ in $R_a(\Gamma)$ in the definition of $\cI(\alpha, \tau)$ (and $\cJ(\alpha, \tau))$ 
without changing the conclusions and proofs of Main Theorem \ref{MThm:trace}.
However we employ the present definition as
it looks more natural, leads to a complete statement (Lemma \ref{Lem:extension}),
 and we are not aware of any important applications of this formal generalization.
\end{Rem}

\subsection*{Acknowledgements}
The author is grateful to the referee for careful reading and helpful comments.
This work was supported by JSPS KAKENHI Early-Career Scientists
(No.~19K14550).


\begin{thebibliography}{99}
\bibitem{Bou}A.~ Le Boudec, {\it \Cs-simplicity and the amenable radical.}
Invent.~ Math., {\bf 209} (2017), 159--174.
\bibitem{BKKO}E.~ Breuillard, M.~ Kalantar, M.~ Kennedy, N.~Ozawa,
{\it \Cs-simplicity and the unique trace property for discrete groups.}
Publ.~ Math.~ I.H.\'{E}.S.~ {\bf 126} (2017), 35--71. 
\bibitem{BK}R.~ S.~ Bryder, M.~ Kennedy, {\it Reduced twisted crossed products
over \Cs-simple groups.} Int.~ Math.~ Res.~ Not.~ {\bf 2018} (2018),
no. 6, 1638--1655.
\bibitem{BS} R.~ C.~ Busby, H.~ A.~ Smith, {\it Representations of twisted group algebras.} Trans.~
Amer.~ Math.~ Soc.~ {\bf 149} (1970), 503--537.

\bibitem{CP}J.~Cuntz, G.~K.~Pedersen, {\it Equivalence and traces on \Cs-algebras.} J.~ Funct.~ Anal.~ {\bf 33} (1979), 135--164.
\bibitem{Ech} S.~Echterhoff, {\it Crossed products and the Mackey--Rieffel--Green machine.} Oberwolfach Seminars, vol. {\bf 47} (2017), Birkh\"{a}user/Springer, 5--79.
\bibitem{Fur}A.~ Furman, {\it On minimal strongly proximal actions of locally compact groups.}
Israel J.~ Math.~ {\bf 136} (2003), 173--187.
\bibitem{Gre}P.~Green, {\it The local structure of twisted covariance algebras.} Acta Math.~ {\bf 140} (1978),
no. 3-4, 191--250.
\bibitem{Haa}U.~ Haagerup, {\it A new look at \Cs-simplicity and the unique trace property
of a group.} Operator Algebras and Applications, Springer
International Publishing, 2016, 167--176.
\bibitem{Ham85}M.~ Hamana, {\it Injective envelopes of \Cs-dynamical systems.} Tohoku Math.~ J.~ (2) {\bf 37} (1985), 463--487.
\bibitem{Har85}P.~de la Harpe, {\it Reduced \Cs-algebras of discrete groups which are simple with a unique trace.} Lecture Notes in Mathematics 1132, 230--253, Springer, 1985.
\bibitem{Har} P. de la Harpe, {\it On simplcity of reduced \Cs-algebras of groups.} Bull. Lond. Math. Soc. {\bf 39} (2007), 1--26.
\bibitem{KK}M.~ Kalantar, M.~ Kennedy, {\it Boundaries of reduced \Cs-algebras of discrete groups.} J.~ reine angew.~ Math.~ {\bf 727} (2017), 247--267.
\bibitem{Ken}M.~ Kennedy, {\it An intrinsic characterization of \Cs-simplicity.} Ann. Sci. \'{E}c. Norm. Sup{\'e}r.,(4), {\bf 53}(5) (2020), 1105--1119.

\bibitem{Lan}E.~ C.~ Lance, {\it Hilbert \Cs-modules: a toolkit for operator algebraists.} LMS Lecture Note Series {\bf 210}, Cambridge University Press, Cambridge, 1995.
\bibitem{PR}J.~ A.~ Packer, I.~ Raeburn, {\it Twisted crossed products of \Cs-algebras.} Math.~ Proc.~Cambridge Philos.~ Soc.~ {\bf 106} (1989), no. 2, 293--311.
\bibitem{Ped}G.~ K.~ Pedersen, {\it Measure theory for \Cs-algebras.} Math.~ Scand.~ {\bf 19} (1966), 131--145.

\bibitem{Pedbook}G.~Pedersen, {\it \Cs-algebras and their automorphism groups.} 
Pure and Applied Mathematics, Academic Press, London, 2018. Second edition.
\bibitem{Rau}S. Raum, {\it \Cs-simplicity of locally compact Powers groups.}
J.~ reine angew.~ Math.~ {\bf 748} (2019), 173--205.
\bibitem{Rau2}S.~Raum, {\it Erratum to \Cs-simplicity of locally compact Powers groups} (J.~ reine angew.~ Math. 748 (2019), 173--205).
J.~ reine angew.~ Math. {\bf 772} (2021), 223--225.
\bibitem{Suzsim}Y.~Suzuki, {\it Elementary constructions of non-discrete \Cs-simple groups.}
Proc.~ Amer.~ Math.~ Soc.~ {\bf 145} (2017), 1369--1371.
\bibitem{SuzMA}Y.~Suzuki, {\it Non-amenable tight squeezes by Kirchberg algebras.}
Math.~ Ann.~ {\bf 382} (2022), 631--653.
\bibitem{Suz}Y.~Suzuki, {\it \Cs-simplicity has no local obstruction.}
Forum Math.~ Sigma {\bf 10} (2022) e18, 8 pages.
\end{thebibliography}
\end{document}